\documentclass[12pt]{amsart}
\usepackage[headings]{fullpage}
\usepackage{amssymb,epic,eepic,epsfig,amsbsy,amsmath,amscd,bm,overpic}
\usepackage{url}
\usepackage[all]{xy}
\numberwithin{equation}{section}
                        \theoremstyle{plain}
\usepackage{mathrsfs}

\newcommand\no[1]{}

\newtheorem{theorem}{Theorem}[section]
\newtheorem{thm}{Theorem}
\newtheorem{lemma}[theorem]{Lemma}

\newtheorem{proposition}[theorem]{Proposition}

\theoremstyle{definition}

\def\BC{\mathbb C}
\def\BN{\mathbb N}
\def\BZ{\mathbb Z}
\def\BR{\mathbb R}

\def\BQ{\mathbb Q}

\def\la{\langle}
\def\ra{\rangle}

\DeclareMathOperator{\tr}{\mathrm tr}

\def\be { \begin{equation} }
\def\ee { \end{equation} }

\begin{document}

\title[Classical pretzel knots and Left orderability]{Classical pretzel knots and Left orderability}

\author{Arafat Khan}
\author{Anh T. Tran}


\begin{abstract}
  We consider the classical pretzel knots $P(a_1, a_2, a_3)$, where $a_1, a_2, a_3$ are positive odd  integers. By using continuous paths of elliptic $\mathrm{SL}_2(\BR)$-representations, we show that (i) the 3-manifold obtained by $\frac{m}{l}$-surgery on $P(a_1, a_2, a_3)$ has left orderable fundamental group if $\frac{m}{l} < 1$, and (ii) the $n^{\mathrm{th}}$-cyclic branched cover of $P(a_1, a_2, a_3)$ has left orderable fundamental group if $n > 2\pi / \arccos(1-2/(1+a_1 a_2 + a_2 a_3 + a_3 a_1))$. 
\end{abstract}

\address{Department of Mathematical Sciences, The University of Texas at Dallas, 
Richardson, TX 75080, USA}
\email{arafat@utdallas.edu}

\address{Department of Mathematical Sciences, The University of Texas at Dallas, 
Richardson, TX 75080, USA}
\email{att140830@utdallas.edu}

\maketitle

\section{Introduction}

A non-trivial group $G$ is called left orderable if it admits a total ordering $<$ such that $g<h$ implies $fg < fh$ for all elements $f,g,h$ in $G$. We study the left orderability of the fundamental groups of $3$-manifolds. This study is motivated by  the L-space conjecture of Boyer-Gordon-Watson \cite{BGW} which states that an irreducible rational homology 3-sphere is an L-space if and only if its fundamental group is not left orderable. Here a rational homology 3-sphere $Y$ is an L-space if its Heegaard Floer homology $\widehat{\mathrm{HF}}(Y)$ has rank equal to the order of $H_1(Y; \BZ)$ \cite{OS}.  We will focus on $3$-manifolds obtained by Dehn surgeries along a knot in $S^3$ or by taking cyclic branched covers of a knot.

For a rational knot $C(k,l)$ in the Conway notation or the $(-2,3,7)$-pretzel knot, some intervals of slopes for which the 3-manifold  obtained from $S^3$ by Dehn surgery along the knot has left orderable fundamental group was determined in \cite{BGW, HT-52, HT-genus1, Tr, Tr-I, Ga, KTT, Va} by using continuous paths of elliptic and hyperbolic $\mathrm{SL}_2(\BR)$-representations of the knot group and the fact that the universal covering group of $\mathrm{SL}_2(\BR)$ is a left orderable group. In the case of the figure eight knot, by using taut foliations it was proved in \cite{Zu} that any non-trivial surgery on the $n^{\text{th}}$-cyclic branched cover of the figure-eight knot has left orderable fundamental group, see also \cite{Hu-taut}. 

A sufficient condition for the fundamental group of the $n^{\text{th}}$-cyclic branched cover of $S^3$ along a prime knot to be left orderable was given in \cite{BGW, Hu} in terms of $\mathrm{SL}_2(\BR)$-representations of the knot group. As an application, it was proved in \cite{Go} that for any rational knot $K$ with non-zero signature the fundamental group of the $n^{\text{th}}$-cyclic branched cover of $S^3$ along $K$ is left orderable for sufficiently large $n$, see also \cite{Hu, Tr-sign}. For a rational knot  $C(k,l)$ or  $C(2n+1,2,2)$ in the Conway notation, the left orderability of the fundamental groups of the cyclic branched covers of $S^3$ along the knot was also determined in \cite{DPT, Tr-cover, Tu}. 

A systematic approach to the understanding of the left orderability of the fundamental groups of 3-manifolds obtained by Dehn surgeries along a knot in $S^3$ or by taking cyclic branched covers of a knot was proposed by Culler-Dunfield \cite{CD}, by using continuous paths of elliptic $\mathrm{SL}_2(\BR)$-representations of the knot group. A similar approach using hyperbolic $\mathrm{SL}_2(\BR)$-representations was also proposed by Gao \cite{Ga-holonomy}. A geometric approach using taut foliations was proposed in \cite{Zu} and references therein.

In this paper, we consider the classical pretzel knots $P(a_1, a_2, a_3)$, where $a_1, a_2, a_3$ are positive odd  integers. By using continuous paths of elliptic $\mathrm{SL}_2(\BR)$-representations, we will prove the following. 

\begin{thm} \label{main}
Let $a_1, a_2, a_3$ be positive odd  integers. Then 

(i) the 3-manifold obtained by $\frac{m}{l}$-surgery on $P(a_1, a_2, a_3)$ has left orderable fundamental group if $\frac{m}{l} < 1$, and

(ii) the $n^{\mathrm{th}}$-cyclic branched cover of $P(a_1, a_2, a_3)$ has left orderable fundamental group if $n > 2\pi / \arccos(1-2/(1+a_1 a_2 + a_2 a_3 + a_3 a_1))$. 
\end{thm}

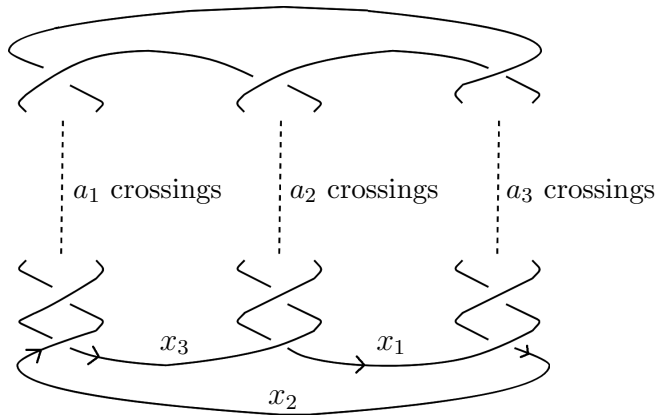
\begin{figure}[h]
	\centering
    \begin{overpic}[width=0.45 \textwidth,tics=9]{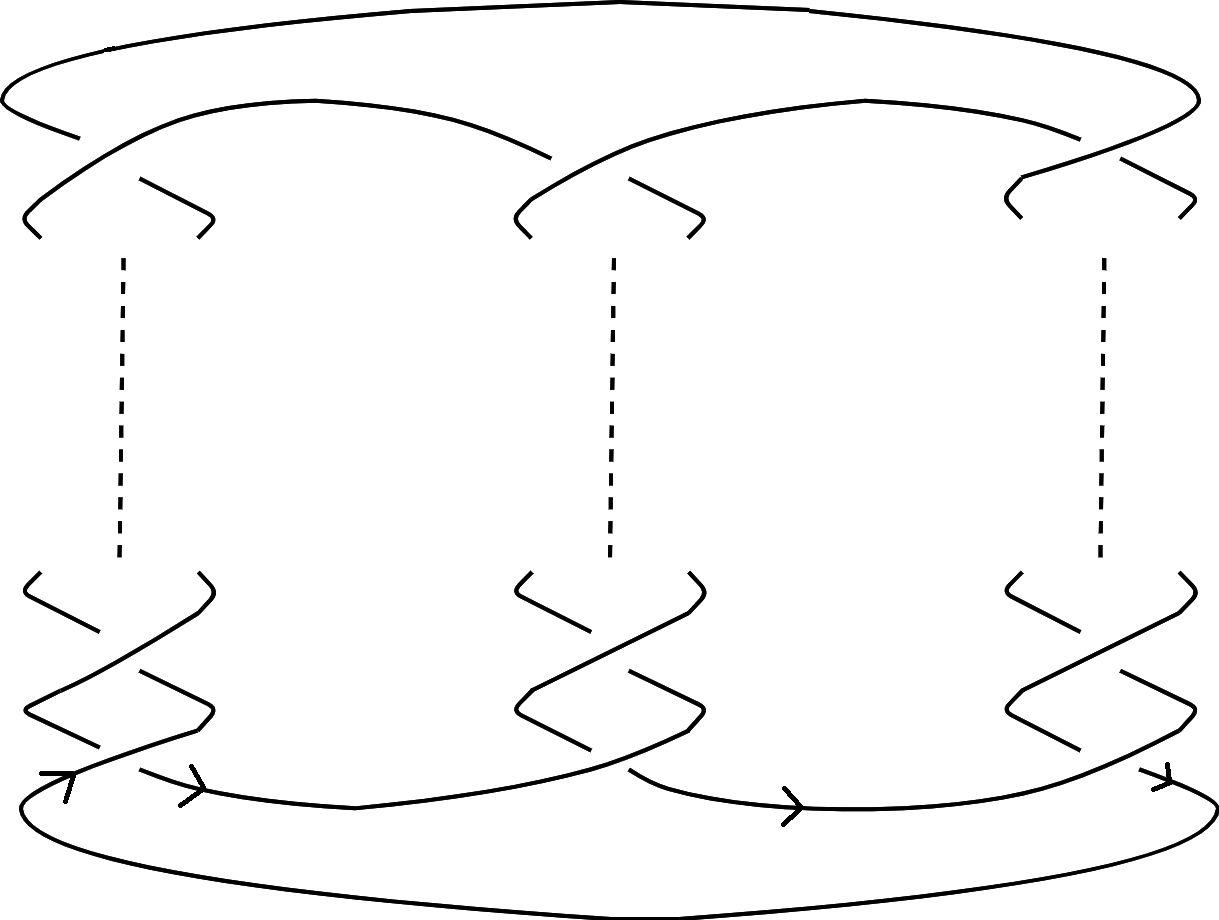}
        \put (12, 40) {\small{\text{$a_1$ crossings}}}
        \put (52, 40) {\small{\text{$a_2$ crossings}}}
        \put (92, 40) {\small{\text{$a_3$ crossings}}}
        \put (68, 12) {{$x_1$}}
        \put (48, 2) {{$x_2$}}
        \put (28, 12) {{$x_3$}}
    \end{overpic}
    	\caption{The pretzel knot $P(a_1, a_2, a_3)$.}
	\label{fig:pretzel}
\end{figure}

The pretzel knot $P(a_1, a_2, a_3)$, where $a_1, a_2, a_3$ are positive odd  integers, is a non-torus alternating knot and hence any non-trivial surgery along it yields a 3-manifold which is not an  L-space, see \cite{OS}. Moreover, Boileau-Boyer-Gordon \cite{BBG} proved that the $n^{\text{th}}$-cyclic branched cover of $P(a_1, a_2, a_3)$ is not an L-space if $n > 2\pi / \arccos(1-2/(1+a_1 a_2 + a_2 a_3 + a_3 a_1))$. So our results support the L-space conjecture.

The paper is organized as follows. In Section \ref{rep} we recall a description of $\mathrm{SL}_2(\BC)$-representations of odd classical pretzel knots and then we study $\mathrm{SL}_2(\BR)$-representations of these knots.  In Section \ref{LO} we use continuous families of $\mathrm{SL}_2(\BR)$-representations  constructed in Section \ref{rep} to give a proof of Theorem \ref{main}. 

\section{$\mathrm{SL}_2(\BR)$-representations of pretzel knots} \label{rep}

In this section we study   $\mathrm{SL}_2(\BR)$-representations of pretzel knots $P(a_1, a_2, a_3)$ where $a_1, a_2, a_3$ are positive odd integers. We first recall the Chebyshev polynomials of the second kind and prove their properties. Then we recall a description of  $\mathrm{SL}_2(\BC)$-representations of $P(a_1, a_2, a_3)$ in \cite{Ha}. Finally, we prove the existence and continuity of a family of  $\mathrm{SL}_2(\BR)$-representations. We also prove properties of this family. 

\subsection{Chebychev polynomials}

Let $\{S_j(z)\}_{j \in \BZ}$ be the sequence of Chebychev polynomials in the variable $z$ defined by $S_0(z) = 1$, $S_1(z) = z$ and $S_j(z) = z S_{j-1}(z) - S_{j-2}(z)$ for all integers $j$. By induction we have $S_j(\pm2) = (\pm 1)^j (j+1)$ and $S_j(z) = \frac{s^{j+1} - s^{-(j+1)}}{s - s^{-1}}$ for $z = s + s^{-1} \not= \pm 2$. Using this fact, one can prove the following, see e.g. \cite{KTT}.

\begin{lemma} \label{chev}
We have $S^2_n(z) - S_{n+1}(z) S_{n-1}(z) = S^2_n(z) + S^2_{n-1}(z) - z S_{n}(z) S_{n-1}(z) = 1$.
\end{lemma}

\begin{lemma} \label{fix_n}
Fix $z \ge 2$. Then $\frac{S_{n}(z)}{S_{n-1}(z)}$ is a strictly decreasing function in $n \in \BN$.
\end{lemma}

\begin{proof}
By Lemma \ref{chev} we have $S^2_n(z) - S_{n+1}(z) S_{n-1}(z) = 1$. Hence 
$\frac{S_{n}(z)}{S_{n-1}(z)} - \frac{S_{n+1}(z)}{S_{n}(z)}= \frac{1}{S_{n}(z)S_{n-1}(z)} > 0
$ for $z \ge 2$. The lemma follows. 
\end{proof}

\begin{lemma} \label{p}
Fix $n \in \BN$. Let $p(z) = \frac{S_{n}(z)}{S_{n-1}(z)}$ for $z \in (2,\infty)$. Then both $p$ and $\frac{z-2}{p-1}$ are both strictly increasing functions in $z$. 
\end{lemma}

\begin{proof} Write $z=s+s^{-1}$ where $s=\frac{1}{2}(z+\sqrt{z^2-4}) >1$. Then $p(z) = \frac{S_{n}(z)}{S_{n-1}(z)}=\frac{s^{n+1}-s^{-n-1}}{s^n-s^{-n}}$.

Note that $\frac{dp}{dz} = \frac{dp}{ds} \big/ \frac{dz}{ds} =  \frac{dp}{ds} \big/(1-s^{-2})$. To prove $p(z)$ is stricly increasing in $(2,\infty)$, it suffices to show $g(s):=\frac{s^{n+1}-s^{-n-1}}{s^n-s^{-n}}$ is a strictly increasing function in $(1,\infty)$. We have
\begin{eqnarray*}
s (s^n-s^{-n})^2 g'(s) &=&(n+1)(s^{n+1} + s^{-n-1})(s^n-s^{-n}) - n (s^{n+1}-s^{-n-1}) (s^{n} + s^{-n}) \\
&=& (s^{n+1} + s^{-n-1})(s^n-s^{-n}) - 2n (s-s^{-1}) \\
&>& 2(s^n-s^{-n}) - 2n (s-s^{-1}) >0.
\end{eqnarray*}
Here we used  $\frac{s^n-s^{-n}}{s-s^{-1}} = S_{n-1}(z) > S_{n-1}(2)=n$ in the last inequality. Hence $g'(s)>0$.

By a direct calculation we have 
$$
\frac{z-2}{p-1} =\frac{s^{-1}(s-1)^2}{(s-1)(s^n + s^{-n-1})/(s^{n}-s^{-n})} = \frac{(s-1)(s^{2n}-1)}{s^{2n+1} + 1} = 1 - \frac{s^{2n}+s}{s^{2n+1} + 1}.
$$
To prove $\frac{z-2}{p-1}$ is stricly increasing in $(2,\infty)$, it suffices to show that $h(s):=-\frac{s^{2n}+s}{s^{2n+1} + 1}$ is a strictly increasing function in $(1,\infty)$. 
We have
\begin{eqnarray*}
(s^{2n+1} + 1)^2 h'(s) &=&-(2n s^{2n-1}+1)(s^{2n+1}+1) + (2n+1)(s^{2n}+s) s^{2n} \\
&=& (s^{4n}-1) + 2n (s^2-1)s^{2n-1} > 0.
\end{eqnarray*}
 Hence $h'(s)>0$ and the lemma follows.
 \end{proof}

\subsection{$\bm{\mathrm{SL}_2(\BC)}$-representations} For a knot $K$ in $S^3$, let $G(K)$ denote the knot group of $K$ which is the fundamental group of the complement of an open tubular neighborhood of $K$. 

It is known that a pretzel link $P(a_1, a_2, \cdots, a_l)$ is isotopic to any pretzel link obtained by a cyclic permutation or reversing the order of $a_1, a_2, \cdots, a_l$. In the case $l=3$, this implies that $P(a_1, a_2, a_3)$ is isotopic to $P(a_{f(1)}, a_{f(2)}, a_{f(3)})$ for any bijection $f: \{1,2,3\} \to \{1,2,3\}$. 

From now on we consider the pretzel knot $P(a_1, a_2, a_3)$ where $a_1, a_2, a_3$ are positive odd integers. By the above argument we can always assume that $a_1 \le a_2 \le a_3$. 

Write $a_j = 2k_j +1$ for $j=1, 2, 3$.  An application of the Wirtinger algorithm to the knot diagram of $K :=P(a_1, a_2, a_3)$ in Figure \ref{fig:pretzel} shows that the knot group $G(K)$ has a presentation with 3 generators  $x_1, x_2, x_3$ and 3 relations
\begin{eqnarray*}
(x_2 x_3^{-1})^{k_1+1} x_3 (x_2 x_3^{-1})^{-k_1-1} &=& (x_1 x_2^{-1})^{k_3} x_1 (x_1 x_2^{-1})^{-k_3}, \\
(x_3 x_1^{-1})^{k_2+1} x_1 (x_3 x_1^{-1})^{-k_2-1} &=& (x_2 x_3^{-1})^{k_1} x_2 (x_2 x_3^{-1})^{-k_1}, \\
(x_1 x_2^{-1})^{k_3+1} x_2 (x_1 x_2^{-1})^{-k_3-1} &=& (x_3 x_1^{-1})^{k_2} x_3 (x_3 x_1^{-1})^{-k_2},
\end{eqnarray*}
where $x_1, x_2, x_3$ are the meridians depicted in Figure \ref{fig:pretzel}. Note that we actually just need 2 of the above 3 relations.

For a  representation $\rho: \la x_1, x_2, x_3 \ra \to \mathrm{SL}_2(\BC)$ we let $t_j :=\tr \rho(x_j)$, $r_1:= \tr \rho(x_2 x^{-1}_3)$, $r_2 := \tr \rho(x_3 x^{-1}_1)$, $r_3 := \tr \rho(x_1 x^{-1}_2)$ and $r :=\tr \rho(x_1 x_2 x_3)$. Let 
$$\sigma_1 := r_1 + r_2 + r_3, \quad \sigma_2 := r_1r_2 + r_2r_3 + r_3r_1, \quad \sigma_3 := r_1r_2r_3.$$ 
Note that if $\rho: \la x_1, x_2, x_3 \ra \to \mathrm{SL}_2(\BC)$ extends to a representation $\rho: G(K) \to \mathrm{SL}_2(\BC)$ then we must have $t_1=t_2=t_3$, since $x_1, x_2, x_3$ are conjugate to each other in $G(K)$.  

By \cite[Theorem 3.12]{Ha} a representation $\rho: \la x_1, x_2, x_3 \ra \to \mathrm{SL}_2(\BC)$, with $t_1 = t_2 = t_3 \not=0$ and $\sigma_1+2-2\gamma \not=0$, extends to an irreducible representation $\rho: G(K) \to \mathrm{SL}_2(\BC)$ if the following equations are satisfied
\begin{align}
(\gamma-2-r_j)S_{k_j}(r_j) &= (\sigma_1-r_j-\gamma)S_{k_j-1}(r_j), \quad j=1,2,3, \label{eq:Af1}\\
t^2 \left( \gamma^2-(\sigma_1+2)\gamma+\sigma_2+4 \right) &= 4+\sigma_3-2\sigma_2- \sigma_1^2, \label{eq:Af2}
\end{align}
where $t:= t_1$ and $\gamma := t^2 +1 - r/t$.

\subsection{$\bm{\mathrm{SL}_2(\BR)}$-representations} We  now study $\mathrm{SL}_2(\BR)$-representations of the pretzel knot $P(a_1, a_2, a_3)$ by finding real solutions of the system of equations \eqref{eq:Af1}--\eqref{eq:Af2}.

\subsubsection{Solving $r_1, r_2, r_3, \gamma$} We first consider real solutions of the system \eqref{eq:Af1}. Moreover we are interested in the case $r_1, r_2, r_3 \in (2, \infty)$. 

Let $p_j := \frac{S_{k_j}(r_j)}{S_{k_j-1}(r_j)}$ for  $j=1, 2, 3$. Then $p_j >1$ and  \eqref {eq:Af1} becomes $\gamma = \frac{\sigma_1 - r_j + (r_j+2) p_j}{p_j+1}$ for $j=1,2,3$. Hence the triple $(r_1, r_2, r_3) \in (2,\infty)^3$ satisfies 
$$
\frac{r_2+r_3+(r_1+2)p_1}{p_1+1} = \frac{r_3+r_1+(r_2+2)p_2}{p_2+1} = \frac{r_1+r_2+(r_3+2)p_3}{p_3+1}.
$$
It is easy to check that this system is equivalent to 
\begin{align}
&(r_1-2)(p_2-p_3)=(r_2-r_3)(p_2p_3-1), \label{eq:r1}\\
&(r_2-2)(p_1-p_3)=(r_1-r_3)(p_1p_3 -1), \label{eq:r2}\\
&(r_3-2)(p_1-p_2)=(r_1-r_2)(p_1p_2-1). \label{eq:r3}
\end{align}
For example, the equation $\frac{r_2+r_3+(r_1+2)p_1}{p_1+1} = \frac{r_3+r_1+(r_2+2)p_2}{p_2+1}$ is equivalent to \eqref{eq:r3}. Note that we actually just need two of the three equations \eqref{eq:r1}--\eqref{eq:r3}.

\begin{proposition} \label{existence}
Suppose $k_1 < k_3$. Then for every $r_1 \in (2, \infty)$, there exist a unique  $r_3 \in (2, r_1)$ such that the triple $(r_1, r_2, r_3)$, where $r_2 = 2+(r_1-r_3)\frac{p_1p_3-1}{p_1-p_3}$, satisfies equations \eqref{eq:r1}--\eqref{eq:r3}. Moreover $r_3$ is a  continuous function in $r_1 \in (2, \infty)$. 
\end{proposition}

\begin{proof}
Fix $r_1 \in (2, \infty)$. Consider $r_3 \in (2, r_1)$. Since $1 \le k_1 < k_3$ and $2 < r_3 < r_1$, by Lemmas \ref{fix_n}--\ref{p} we have $p_1 = \frac{S_{k_1}(r_1)}{S_{k_1-1}(r_1)}  > \frac{S_{k_3}(r_3)}{S_{k_3-1}(r_3)} = p_3>1$.  From \eqref{eq:r2} we get
$$r_2 = 2+(r_1-r_3)\frac{p_3p_1-1}{p_1-p_3}>2.$$

Since $r_1 - r_2 = r_1 - 2 - (r_1-r_3)\frac{p_1p_3-1}{p_1-p_3} = r_3 -2 - (r_1-r_3)\left( \frac{p_1p_3-1}{p_1-p_3} -1 \right) = r_3 -2 - (r_1-r_3)  \frac{(p_1+1)(p_3-1)}{p_1-p_3}$,  equation \eqref{eq:r3} then becomes
\begin{align*}
&(r_3-2)(p_1-p_2)=\left( r_3 -2 - (r_1-r_3)  \frac{(p_1+1)(p_3-1)}{p_1-p_3} \right)(p_1p_2-1)\\
\iff & (r_1-r_3)\frac{(p_1+1)(p_3-1)(p_1p_2-1)}{p_1 - p_3}-(r_3-2)(p_1+1)(p_2-1)=0 \\
\iff& (r_1-r_3)\frac{p_1p_2-1}{p_2-1}- (r_3-2) \frac{p_1 - p_3}{p_3 -1} =0\\
\iff& (r_1-r_3)(\frac{(p_1-1)p_2}{p_2-1} + 1)  -(r_3-2) (\frac{p_1 - 1}{p_3 -1}- 1) =0\\
\iff& (r_1-r_3)\frac{(p_1-1)p_2}{p_2-1}   -(r_3-2) \frac{p_1 - 1}{p_3 -1} + r_1 -2 =0\\
\iff&\frac{r_3-2}{p_3-1}-(r_1-r_3)\left( 1 + \frac{1}{p_2-1} \right) -\frac{r_1-2}{p_1-1} =0.
\end{align*}

Consider the  function
$f:(2,r_1)\to \mathbb{R}$
defined by
$$f(r_3) =   \frac{r_3-2}{p_3-1} - (r_1-r_3) - \frac{r_1 - r_3}{p_2 - 1} - \frac{r_1-2}{p_1-1}.$$
We claim that $f$ is a strictly increasing function in $r_3 \in (2,r_1)$. Indeed, we have
\begin{align*}
\frac{df}{dr_3} &= \frac{d}{dr_3} \left(\frac{r_3-2}{p_3-1} \right) + 1 -\frac{d}{d r_3}\left(\frac{r_1-r_3}{p_2-1}\right) \\
&=  \frac{d}{dr_3} \left(\frac{r_3-2}{p_3-1} \right) + 1 + \frac{(r_1-r_3)\frac{d p_2}{d r_2}\frac{d r_2}{d r_3}+(p_2-1)}{(p_2-1)^2},
\end{align*}
and
\begin{align*}
\frac{d r_2}{d r_3}&=\frac{d}{d r_3}\left(2+(r_1-r_3)\frac{p_1p_3 - 1}{p_1-p_3}\right)\\
&= (r_1-r_3)\frac{d}{d r_3}\left(\frac{p_1p_3 - 1}{p_1-p_3}\right) -\frac{p_1p_3 - 1}{p_1-p_3}\\
&=\frac{(r_1-r_3)(p^2_1-1)}{(p_1-p_3)^2}  \frac{d p_3}{d r_3} -\frac{r_2-2}{r_1-r_3}.
\end{align*}
Hence, by noting that $\frac{(p_2-1) - (r_2-2)\frac{d p_2}{d r_2}}{(p_2-1)^2} = \frac{d}{dr_2}\left( \frac{r_2-2}{p_2-1} \right)$, we obtain 
$$
\frac{df}{dr_3} 
=  \frac{d}{dr_3} \left(\frac{r_3-2}{p_3-1} \right) + 1 + \frac{(r_1 - r_3)^2 (p^2_1-1)}{(p_1-p_3)^2(p_2-1)^2}  \frac{d p_3}{d r_3}  \frac{d p_2}{d r_2} +\frac{d}{dr_2}\left( \frac{r_2-2}{p_2-1} \right)>1.
$$
Here we used the facts that $\frac{d p_2}{d r_2} > 0$, $\frac{d p_3}{d r_3}>0$, $\frac{d}{dr_2}\left( \frac{r_2-2}{p_2-1} \right)>0$ and $\frac{d}{dr_3}\left( \frac{r_3-2}{p_3-1} \right)>0$ in the last inequality. These facts follow from Lemma \ref{p}. 

We have proved that $f(r_3)$ is a strictly increasing function in $r_3 \in (2,r_1)$. Note that
\begin{eqnarray*}
\lim_{r_3 \to 2^+} f(r_3) &= -(r_1-2) (\frac{p_2}{p_2-1}+\frac{1}{p_1-1})< 0, \\
 \lim_{r_3 \to r_1^-} f(r_3) &= (r_1-2) (\frac{1}{p_3-1}-\frac{1}{p_1-1})>0.
 \end{eqnarray*}
 Hence there exists a unique $r_3 \in (2,r_1)$ such that $f(r_3)=0$. By the Implicit Function Theorem, $r_3$ is a continuous function in $r_1 \in (2, \infty)$. 
\end{proof}

From now on, we consider the triple $(r_1, r_2, r_3) \in (2,\infty)^3$ in Proposition \ref{existence} when 
$k_1 < k_3$. If $k_1=k_3$, by letting $r_1=r_2=r_3 > 2$ the equations \eqref{eq:r1}-\eqref{eq:r3} are satisfied. 

\begin{lemma} \label{triangle}
We have $r_1 + r_2 > r_3 +2$, $r_2 + r_3 > r_1 +2$ and $r_3 + r_1 > r_2 +2$.
\end{lemma}

\begin{proof}
The case $k_1=k_3$ is obvious, so we only consider the case $k_1 < k_3$. 
Note that if $x,y>1$ then $\left|\frac{xy-1}{x-y} \right| >1$. Since $r_2 = 2 + (r_1-r_3)\frac{p_1p_3-1}{p_1-p_3}$ we have $r_2 - 2 = |r_1-r_3| \left|\frac{p_1p_3-1}{p_1-p_3} \right|>|r_1-r_3|$. This implies that $r_2 + r_1 > r_3 +2$ and $r_2 + r_3 > r_1 +2$.

Since $p_1 \not= p_3$ we have either $p_1 \not= p_2$ or $p_3 \not= p_2$. If $p_1 \not= p_2$ then  $r_3 - 2 = |r_1-r_2| \left|\frac{p_1p_2-1}{p_1-p_2} \right|>|r_1-r_2|$, which implies that $r_3 + r_1 > r_2+2$. 
Similarly, if $p_1 \not= p_3$ then  $r_1 - 2 = |r_2-r_3| \left|\frac{p_2p_3-1}{p_2-p_3} \right|>|r_2-r_3|$, which also implies that $r_3 + r_1 > r_2+2$. 
\end{proof}

\begin{lemma} \label{ratio}
We have
\begin{eqnarray*}
\lim_{r_1 \to 2^+} \frac{r_2 -2}{r_1 -2} &=& \frac{1+k_1 + k_3}{1+k_2 + k_3}, \\
\lim_{r_1 \to 2^+}  \frac{r_3 -2}{r_1 -2}  &=& \frac{1+k_1 + k_2}{1+k_2 + k_3}.
\end{eqnarray*}
\end{lemma}

\begin{proof}
The case $k_1=k_3$ is obvious, so it suffices to consider the case $k_1 < k_3$. Since $r_3 \in (2, r_1)$ and $r_1 + r_3 -2 > r_2 >2$, we have $r_3, r_2 \to 2^+$ as $r_1 \to 2^+$. This implies that $p_j \to \frac{S_{k_j}(2) }{S_{k_j -1}(2) }= \frac{k_j +1}{k_j} = 1 +  \frac{1}{k_j}$ as $r_1 \to 2^+$. Hence
\begin{equation} \label{eq1}
\lim_{r_1 \to 2^+} \frac{r_1 - r_2}{r_3 -2} = \lim_{r_1 \to 2^+} \frac{p_1-p_2}{p_1p_2-1} =  \frac{k_2 - k_1}{k_1 + k_2 +1}. 
\end{equation}
Similarly, we have 
\begin{equation}
\lim_{r_1 \to 2^+} \frac{r_1 - r_3}{r_2 -2} =  \frac{k_3 - k_1}{k_1 + k_3 +1}. \label{eq2}
\end{equation}

Let $q_2 = \lim_{r_1 \to 2^+} \frac{r_2 -2}{r_1 -2}$ and $q_3 = \lim_{r_1 \to 2^+} \frac{r_3 -2}{r_1 -2}$. From \eqref{eq1} and \eqref{eq2} we have 
$$
\frac{1 - q_2}{q_3} =  \frac{k_2 - k_1}{k_1 + k_2 +1}
\qquad \text{and} \qquad
\frac{1 - q_3}{q_2} =  \frac{k_3 - k_1}{k_1 + k_3 +1}.
$$
Hence 
$
q_2 = \frac{1+k_1 + k_3}{1+k_2 + k_3}$ 
and  $q_3  = \frac{1+k_1 + k_2}{1+k_2 + k_3}$ as claimed. 
\end{proof}

Let $\gamma := \frac{r_2+r_3+(r_1+2)p_1}{p_1+1}$ and $r :=x^3+x-x\gamma$. Then the system \eqref {eq:Af1} is satisfied. 

\begin{lemma} \label{sg}
We have $\sigma_1 +2 - 2\gamma >0$. 
\end{lemma} 

\begin{proof}
This is because
$$
\sigma_1 +2 - 2\gamma = r_2 + r_2 + r_3 +2 - \frac{r_2+r_3+(r_1+2)p_1}{p_1+1}=\frac{(r_2 + r_3-r_1-2)(p_1 - 1)}{p_1+1}.
$$
and $r_2 + r_3-r_1-2 >0$ (by Lemma \ref{triangle}).
\end{proof}

\subsubsection{Trace of meridian} We now study the equation \eqref{eq:Af2}. Recall that $\sigma_1 = r_1 + r_2 + r_3$, $\sigma_2 = r_1r_2 + r_2r_3 + r_3r_1$ and $\sigma_3 = r_1r_2r_3$.   

Let $\delta:=r_1 r_2 r_3+4 -r_1^2- r_2^2 - r_3^2=4+\sigma_3-2\sigma_2- \sigma_1^2$ be the RHS of  \eqref{eq:Af2}. 

\begin{lemma} \label{delta}
We have $\delta > (r_1-2)(r_2-2)(r_3-2)>0$.
\end{lemma}

\begin{proof}
Let $a_j = r_j -2>0$. By Lemma \ref{triangle} we have $a_1+a_2 > a_3$, $a_2 + a_3 > a_1$ and $a_3 + a_1 > a_2$. Hence
\begin{eqnarray*}
\delta &=& (a_1+2)(a_2+2)(a_3+2)+4-(a_1+2)^2-(a_2+2)^2-(a_3+2)^2 \\
&=& a_1 a_2 a_3 + 2a_1 a_2 + 2a_2 a_3 + 2a_3 a_1 - a_1^2 - a_2^2 - a_3^2 \\
&=& a_1 a_2 a_3 + (a_1+a_2-a_3)a_3 + (a_2 + a_3 - a_1)a_1 + (a_3 + a_1-a_2)a_2 \\ &>& a_1a_2 a_3 > 0.
\end{eqnarray*}
The lemma follows.
\end{proof}

For the LHS of \eqref{eq:Af2} we have the following.

\begin{lemma} \label{LHS}
We have 
$$
\gamma^2-(\sigma_1+2)\gamma+\sigma_2+4 = (r_2-2)(r_3-2) - \frac{(r_2+r_3-r_1-2)^2p_1}{(p_1 +1)^2}.
$$
\end{lemma}

\begin{proof}
We have $\gamma- (\sigma_1+2)= \frac{r_2+r_3+(r_1+2)p_1}{p_1+1} - (r_2 + r_3 + r_1 +2)= - \frac{r_1+2+(r_2+r_3)p_1}{p_1+1}$. This implies that
\begin{eqnarray*}
\gamma^2-(\sigma_1+2)\gamma &=&  - \frac{r_1+2+(r_2+r_3)p_1}{p_1+1} \cdot \frac{r_2+r_3+(r_1+2)p_1}{p_1+1} \\
&=& - \left((r_1+2)(r_2+r_3) +\frac{(r_2+r_3-r_1-2)^2 p_1}{(p_1 +1)^2} \right).
\end{eqnarray*}
The lemma then follows, since $\sigma_2 +4 - (r_1+2)(r_2+r_3)=(r_2-2)(r_3-2)$.
\end{proof}

Since $S^2_{k_1}(r_1) + S^2_{k_1-1}(r_1) =  r_1 S_{k_1}(r_1)S_{k_1-1}(r_1) +1$ (by Lemma \ref{chev}) we have  $$\frac{(p_1+1)^2}{p_1} = \frac{S_{k_1}(r_1)}{S_{k_1-1}(r_1)} + \frac{S_{k_1-1}(r_1)}{S_{k_1}(r_1)} + 2 = \frac{1}{S_{k_1}(r_1)S_{k_1-1}(r_1)} + r_1 +2 > r_1 +2.$$
This implies that $\frac{p_1}{(p_1 +1)^2} < \frac{1}{r_1 +2}$. Hence, by Lemma \ref{LHS} we obtain
\begin{equation} \label{LHS-ineq}
\gamma^2-(\sigma_1+2)\gamma+\sigma_2+4 > (r_2-2)(r_3-2) - \frac{(r_2+r_3-r_1-2)^2}{r_1+2} = \frac{\delta}{r_1+2} > 0. 
\end{equation} 

Let $T =  \frac{\delta}{\gamma^2-(\sigma_1+2)\gamma+\sigma_2+4} > 0$ and $t = \pm \sqrt{T}$. Then the equation \eqref{eq:Af2} is satisfied. 

\begin{lemma} \label{t^2}
We have $r_1 -2 < T < r_1 +2$. 
\end{lemma}

\begin{proof}
Note that $T < r_1 +2$ by  \eqref{LHS-ineq}. Since $0 < \gamma^2-(\sigma_1+2)\gamma+\sigma_2+4 < (r_2-2)(r_3-2)$ (by Lemma \ref{LHS}) and $\delta >  (r_1 -2)(r_2-2)(r_3-2) > 0$ (by Lemma \ref{delta}) we have $T > r_1 -2$. 
\end{proof}

\begin{proposition} \label{lim}
We have $\lim_{r_1 \to \infty}  T = \infty$ and 
$$
\lim_{r_1 \to 2^+} T = 4 - \frac{1}{1 + k_1 + k_2 + k_3 + k_1 k_2 + k_2 k_3 + k_3 k_1}.
$$
\end{proposition}

\begin{proof}
Since $T > r_1 -2$ (by Lemma \ref{t^2}) we have $\lim_{r_1 \to \infty} T = \infty$. 

Let $a_j = r_j -2 >0$ for $j=1,2,3$. By Lemma \ref{ratio} we have 
$\lim_{a_1 \to 0^+} \frac{a_2}{a_1} = q_2$ and 
$\lim_{a_1 \to 0^+}  \frac{a_3}{a_1}  = q_3$,
where $q_2= \frac{1+k_1 + k_3}{1+k_2 + k_3}$ and $q_3 = \frac{1+k_1 + k_2}{1+k_2 + k_3}$. 

From the proof of Lemma \ref{delta} we have $\delta = a_1 a_2 a_3 + 4a_2 a_3 - (a_2 + a_3 -a_1)^2$. By Lemma \ref{LHS} we have $\gamma^2-(\sigma_1+2)\gamma+\sigma_2+4 = a_2 a_3 - (a_2+a_3-a_1)^2/(p_1 + p_1^{-1}+2)$. Hence
\begin{eqnarray*}
\lim_{r_1 \to 2^+} T &=& \lim_{a_1 \to 0^+} \frac{a_1 a_2 a_3 + 4a_2 a_3 - (a_2 + a_3 -a_1)^2}{a_2 a_3 - (a_2+a_3-a_1)^2/(p_1 + p_1^{-1}+2)} \\
&=&\frac{4q_2 q_3 - (q_2 + q_3-1)^2}{q_2 q_3 - (q_2+q_3-1)^2/(\frac{k_1+1}{k_1} + \frac{k_1}{k_1+1} +2)}.
\end{eqnarray*} 

Since $q_2 + q_3 = 1+ \frac{1+2k_1}{1+k_2 + k_3}$ and $q_2q_3 = \frac{(1+k_1 + k_3)(1+k_1 + k_2)}{(1+k_2 + k_3)^2}$ we have 
\begin{eqnarray*}
4q_2 q_3 - (q_2 + q_3-1)^2 
&=& \frac{4(1+k_1 + k_3)(1+k_1 + k_2)}{(1+k_2 + k_3)^2} -  \frac{(2k_1+1)^2}{(1+k_2 + k_3)^2} \\
&=& \frac{3 + 4k_1 + 4k_2 + 4k_3 + 4k_1 k_2 + 4k_2 k_3 + 4k_3 k_1}{(1+k_2+k_3)^2},
\end{eqnarray*} 
and
\begin{eqnarray*}
q_2 q_3 - \frac{(q_2+q_3-1)^2 }{\frac{k_1+1}{k_1} + \frac{k_1}{k_1+1} +2} 
&=& \frac{(1+k_1 + k_3)(1+k_1 + k_2)}{(1+k_2 + k_3)^2}- \frac{k_1(k_1+1)}{(1+k_2+k_3)^2} \\
&=& \frac{1 + k_1 + k_2 + k_3 + k_1 k_2 + k_2 k_3 + k_3 k_1}{(1+k_2+k_3)^2}.
\end{eqnarray*} 
Hence $\lim_{r_1 \to 2^+} T = 4 - 1/(1 + k_1 + k_2 + k_3 + k_1 k_2 + k_2 k_3 + k_3 k_1)$.
\end{proof}

Since $a_j=2k_j+1$, we also have $\lim_{r_1 \to 2^+} T =  4 - 4/(1+a_1 a_2 + a_2 a_3 + a_3 a_1)$. If we let 
$$
\theta_0 = \frac{1}{2}\arccos(1-2/(1+a_1a_2 + a_2 a_3 + a_3 a_1)) \in (0, \frac{\pi}{2})
$$
then $\lim_{r_1 \to 2^+} T = 2 + 2 \cos2\theta = 4 \cos^2 \theta_0$. 

\section{Left-orderability} \label{LO}

In this section we will use the continuous family of irreducible $\mathrm{SL}_2(\BR)$-representations constructed in the previous section to study the left-orderabilty of the fundamental groups of the 3-manifolds obtained by Dehn surgeries on $P(a_1, a_2, a_3)$ and to study the left-orderabilty of the fundamental groups of the cyclic branched covers of $P(a_1, a_2, a_3)$, where $a_1, a_2, a_3$ are positive odd integers. In particular, we will use elliptic $\mathrm{SL}_2(\BR)$-representations to give a proof of Theorem \ref{main}. 

Recall that we have considered the triple $(r_1, r_2, r_3) \in (2, \infty)^3$ in Proposition \ref{existence} when 
$k_1 < k_3$. If $k_1=k_3$ we let $r_1=r_2=r_3>2$. Note that $r_3, r_2$ are continuous functions in $r_1 \in (2, \infty)$. In both cases, for each $r_1 \in (2, \infty)$ we let $r=x^3+x-x\gamma$, $t =\pm \sqrt{T}$ where
$$
\gamma = \frac{r_2+r_3+(r_1+2)p_1}{p_1+1}, \quad T = \frac{\delta}{\gamma^2-(\sigma_1+2)\gamma+\sigma_2+4} >0. 
$$
Note that $\sigma_1 +2 - 2\gamma >0$ by Lemma \ref{sg}. Then the system \eqref{eq:Af1}--\eqref{eq:Af2} is satisfied, and hence there exists an  irreducible representation $\rho^{\pm}=\rho^{\pm}_{r_1}: G(K) \to \mathrm{SL}_2(\BC)$ with $$\left( \tr \rho^{\pm}(x_1),  \tr \rho^{\pm}(x_2 x_3^{-1}), \tr \rho^{\pm}(x_3 x_1^{-1}), \tr \rho^{\pm}(x_1 x_2^{-1}),\tr \rho^{\pm}(x_1 x_2 x_3) \right) = (\pm \sqrt{T}, r_1, r_2, r_3, r).$$ 

Since $\rho^{\pm}$ has real traces, it is conjugate to either an $\mathrm{SL}_2(\BR)$ representation or an $\mathrm{SU}(2)$ representation, see e.g. \cite[Lemma 10.1]{HP}. If the latter occurs, then $|\tr \rho^{\pm}(g)|<2$ for all $g \in G(K)$. Since $\tr \rho^{\pm}_{r_1}(x_2 x_3^{-1}) = r_1 >2$, $\rho^{\pm}$ is conjugate to an $\mathrm{SL}_2(\BR)$ representation.

\subsection{Elliptic $\bm{\mathrm{SL}_2(\BR)}$-representations} 

Let $X$ be the complement of an open tubular neighborhood of $K:=P(a_1, a_2, a_3)$ in $S^3$, and $X_{m/l}$ the 3-manifold obtained from $S^3$ by $\frac{m}{l}$-surgery along $K$.

An element of $\mathrm{SL}_2(\BR)$ is called elliptic if its trace is a real number in $(-2,2)$. A representation $\rho: \BZ^2 \to \mathrm{SL}_2(\BR)$ is called elliptic if the image group $\rho(\BZ^2)$ contains an elliptic element of $\mathrm{SL}_2(\BR)$. In which case, since $\BZ^2$ is an abelian group every non-trivial element of $\rho(\BZ^2)$ must also be elliptic.

\begin{proposition} \label{prop} 
For each rational number $\frac{m}{l}\in(-\infty, 0) \cup (0,1)$ there exists a  representation $\rho: \pi_1(X_{m/l}) \to \mathrm{SL}_2(\BR)$ such that $\rho\big|_{\pi_1(\partial X)}: \pi_1(\partial X) \cong \BZ^2 \to \mathrm{SL}_2(\BR)$ is an elliptic representation. 
\end{proposition}

\begin{proof}
Since $\lim_{r_1 \to 2^+} T = 4\cos^2\theta_0 < 4$ and $\lim_{r_1 \to \infty} T = \infty$ (by Proposition \ref{lim}),   
there exists $r_1^{*} > 2$ such that $T(r_1^{*})=4$ and for all $r_1 \in(2, r_1^{*})$ we have $0 < T(r_1)<4$. 

For each $r_1 \in (2, r_1^*)$ we let $\theta = \arccos (\sqrt{T}/2) \in (0, \pi/2)$.  Note that $\lim_{r_1 \to 2 ^+} \theta = \theta_0$ and $\lim_{r_1 \to r_1^*} \theta = 0$. Then $\tr \rho^{\pm} (x_1)=\sqrt{T} = \pm 2\cos\theta = \pm (e^{i\theta} + e^{-i\theta})$. 

Let $\lambda$ be the canonical longitude corresponding to the meridian $\mu:=x_1$.  Up to conjugation we assume that 
 $$
 \rho^+(\mu)= \left[
\begin{array}{cc}
M & *\\
0 & M^{-1}
\end{array}
\right] \quad \text{and} \quad \rho^+(\lambda) = \left[
\begin{array}{cc}
L & *\\
0 & L^{-1}
\end{array}
\right]
$$
where $M:=e^{i\theta}$. By \cite[Proposition 4.1]{Ha} $L$ satisfies the following equation 
$$
(1+L)(M+M^{-1})(\sigma_1+2-2\gamma) = (1-L)(M-M^{-1})(\sigma_1+2-2t^2).
$$
This can be rewritten as $L(\alpha M- \beta M^{-1}) + \alpha M^{-1}- \beta M=0$, where 
$$
\alpha :=\sigma_1+2-\gamma-t^2, \quad \beta :=\gamma - t^2.
$$
Note that $\alpha, \beta \in \BR$. By Lemmas \ref{triangle} and \ref{t^2} we have
$$
\beta = \frac{(r_1+2)w_1+(r_2 + r_3)v_1}{v_1+w_1} - t^2 = \frac{(r_2 + r_3-r_1-2)v_1}{v_1+w_1} +r_1+2 - t^2> 0.
$$ 
Moreover, we have $\alpha - \beta=\sigma_1 +2 - 2\gamma > 0$ by Lemma \ref{sg}. Hence $\alpha > \beta > 0$. 

We have $|\alpha M- \beta M^{-1}|^2=\alpha^2 + \beta^2 - 2\alpha \beta \cos\theta = (\alpha - \beta)^2 + 2\alpha \beta (1-\cos\theta) >0$, which implies that $\alpha M- \beta M^{-1} \not=0$. Hence $L=-(\alpha M^{-1}- \beta M)/(\alpha M- \beta M^{-1})$. 

Since $\alpha M^{-1}- \beta M$ and $\alpha M- \beta M^{-1}$ are complex conjugates, $L$ is a unit complex number. Moreover, by a direct calculation we have
\begin{eqnarray*}
\text{Re}(L) &=&  \big( 2\alpha \beta -(\alpha^2+\beta^2) \cos 2\theta \big )/(\alpha^2 + \beta^2 - 2\alpha \beta \cos\theta),\\
\text{Im}(L) &=& (\alpha^2 - \beta^2)\sin2\theta/(\alpha^2 + \beta^2 - 2\alpha \beta \cos\theta).
\end{eqnarray*}
Note that $\text{Im}(L) >0$ and so $L = e^{i\varphi}$, where
$$\varphi := \arccos \left( ( 2\alpha \beta -(\alpha^2+\beta^2) \cos 2\theta )/(\alpha^2 + \beta^2 - 2\alpha \beta \cos\theta) \right) \in (0, \pi).$$

As $r_1 \to 2^+$ we have $r_3, r_2 \to 2^+$ and $\theta \to \theta_0$. Note that $\gamma = \frac{r_2+r_3+(r_1+2)p_1}{p_1+1} \to 4$ and so $\alpha, \beta \to 4-4\cos^2\theta_0$. Hence  $L=-(\alpha M^{-1}- \beta M)/(\alpha M- \beta M^{-1}) \to 1$ and $\varphi \to 0$. 

As $r_1 \to (r_1^*)^-$ we have $T \to 4$ and $\theta \to 0^+$. Note that $\alpha \to \alpha(r_1^*)$ and $\beta \to \beta(r_1^*)$, where $\alpha(r_1^*) > \beta(r_1^*) >0$. Hence $L=-(\alpha M^{-1}- \beta M)/(\alpha M- \beta M^{-1}) \to -1$ and  $\varphi \to \pi$. 

Since $\lim_{r_1 \to 2^+} (- \frac{\varphi}{\theta}) = 0$ and $\lim_{r_1 \to (r_1^*)^-} (- \frac{\varphi}{\theta}) = -\infty$, the image of the continuous function $-\frac{\varphi}{\theta}: (2, r_1^*) \to \BR$ contains the interval $(-\infty,0)$. 

We now finish the proof of Proposition \ref{prop}. First, suppose $\frac{m}{l}\in  (-\infty,0) \cap \BQ$. By the above argument, $\frac{m}{l}= - \frac{\varphi(r_1)}{\theta(r_1)}$ for some $r_1 \in (2, r_1^*)$. Since $M^m L^l = e^{i(m\theta + l\varphi)} = 1$, we have $\rho^+(\mu^m \lambda^l) = I$. Let $\rho: G(K) \to \mathrm{SL}_2(\BR)$ be an $\mathrm{SL}_2(\BR)$ representation conjugate to $\rho^+$. Then we also have $\rho(\mu^m \lambda^l) = I$. This means that $\rho$ extends to a  representation $\rho: \pi_1(X_{m/l}) \to \mathrm{SL}_2(\BR)$. Note that $\rho\big|_{\pi_1(\partial X)}$ is an elliptic representation. 

Finally, suppose  $\frac{m}{l}  \in (0,1) \cap \BQ$. Consider $\rho^-=\rho^-(r_1)$ of the form 
$$
 \rho^-(\mu)= \left[
\begin{array}{cc}
M' & *\\
0 & (M')^{-1}
\end{array}
\right] \quad \text{and} \quad \rho^-(\lambda) = \left[
\begin{array}{cc}
L' & *\\
0 & (L')^{-1}
\end{array}
\right]
$$
where $M'=-M= e^{i(\theta-\pi)}$. As above we have $L' = -\frac{\alpha (M')^{-1}- \beta M'}{\alpha M'- \beta (M')^{-1}}= L$ and so $L'=e^{i\varphi}$. 

Since $\lim_{r_1 \to 2^+} (- \frac{\varphi}{\theta - \pi}) = 0$ and $\lim_{r_1 \to (r_1^*)^-} (- \frac{\varphi}{ \theta - \pi}) = 1$, the image of the continuous function $-\frac{\varphi}{\pi - \theta}: (2, r_1^*) \to \BR$ contains the interval $(0,1)$. This implies that $\frac{m}{l}= - \frac{\varphi(r_1)}{\theta(r_1)-\pi}$ for some $r_1 \in (2, r_1^*)$. Since $(M')^m (L')^l = e^{i(m(\theta-\pi) + l\varphi)} = 1$, we have $\rho^-(\mu^m \lambda^l) = I$. 

Let $\rho': G(K) \to \mathrm{SL}_2(\BR)$ be an $\mathrm{SL}_2(\BR)$-representation conjugate to $\rho^-$. Then we also have $\rho'(\mu^m \lambda^l) = I$. Hence $\rho'$ extends to a  representation $\rho': \pi_1(X_{m/n}) \to \mathrm{SL}_2(\BR)$. Note that $\rho' \big|_{\pi_1(\partial X)}$ is an elliptic representation. 
\end{proof}

\subsection{Proof of Theorem \ref{main}} $(i)$ Suppose $\frac{m}{l} \in (-\infty, 1) \cap \BQ$. If $\frac{m}{l} \not= 0$, by Proposition \ref{prop}, there exists a  representation $\rho: \pi_1(X_{m/l}) \to \mathrm{SL}_2(\BR)$ such that $\rho\big|_{\pi_1(\partial X)}$ is an elliptic representation. This representation lifts to a representation $\tilde{\rho}: \pi_1(X_{m/l}) \to \widetilde{\mathrm{SL}_2(\BR)}$, where $\widetilde{\mathrm{SL}_2(\BR)}$ is the universal covering group of $\mathrm{SL}_2(\BR)$. See e.g. \cite[Section 3.5]{CD} and \cite[Section 2.2]{Va}. Note that $X_{m/l}$ is an irreducible 3-manifold (by \cite{HTh}) and $\widetilde{\mathrm{SL}_2(\BR)}$ is a left orderable group (by \cite{Be}). Hence, by \cite{BRW}, $\pi_1(X_{m/l})$ is a left orderable group. 
Finally, $0$-surgery along a knot always produces a prime manifold whose first Betti number is $1$, and by \cite{BRW} such manifold has left orderable fundamental group.

$(ii)$ Suppose $n > 2\pi / \arccos(1-2/(1+a_1a_2 + a_2 a_3 + a_3 a_1))=\pi/\theta_0 $. Since $4\cos^2\frac{\pi}{n} \in (4 \cos^2\theta_0, 4)$, we have $4\cos^2\frac{\pi}{n} = T(r_1)$ for some $r_1 \in (2, r_1^*)$. Consider $\rho^+:=\rho^+_{r_1}$. Then $\tr \rho^+(\mu) = \sqrt{T} = 2\cos\frac{\pi}{n}$. This implies that $e^{\frac{\pm i\pi}{n}}$ are eigenvalues of $\rho^+(\mu)$. Hence $\rho^+(\mu)$ is conjugate to the diagonal matrix $D:=\mathrm{diag}(e^{\frac{i\pi}{n}}, e^{-\frac{i\pi}{n}})$. 

Let $\rho: G(K) \to \mathrm{SL}_2(\BR)$ be an irreducible $\mathrm{SL}_2(\BR)$-representation conjugate to $\rho^+$. Then $\rho(\mu)$  is conjugate to the diagonal matrix $D$. Since $D^n=-I$, we have  $\rho(\mu^n) = -I$. By applying \cite[Theorem 6]{BGW} and \cite[Theorem 3.1]{Hu} we conclude that the fundamendal group of the $n^{\mathrm{th}}$-cyclic branched cover of $P(a_1, a_2, a_3)$ is left-orderable. 

\section*{Acknowledgements} 
This paper is essentially the Ph.D. thesis of the first author written under the supervision of the second author. The second author is partially supported by a grant from the Simons Foundation (\#354595).

\end{document}